\documentclass{amsart}

\usepackage{verbatim}
\usepackage{graphicx}
\usepackage{amssymb}
\usepackage{amsfonts}
\usepackage{amsmath}
\usepackage{amsthm}
\usepackage{hyperref}

\def\RR{\mathbb{R}}
\def\CC{\mathbb{C}}
\def\NN{\mathbb{N}}

\def\FF{\mathbb{F}}

\def\x{\mathbf{x}}
\def\fx{\FF[\x]}

\DeclareMathOperator{\cf}{cf}

\DeclareMathOperator{\hc}{hc}
\DeclareMathOperator{\hd}{hd}
\DeclareMathOperator{\prin}{Prin}
\DeclareMathOperator{\ppar}{Par}
\DeclareMathOperator{\lex}{lex}
\DeclareMathOperator{\card}{card}

\DeclareMathOperator{\hmf}{Hom}

\newtheorem{theorem}{Theorem}
\newtheorem{thm}[theorem]{Theorem}

\newtheorem{lem}[theorem]{Lemma}
\newtheorem{prop}[theorem]{Proposition}
\newtheorem{cor}[theorem]{Corollary}

\theoremstyle{definition}

\newtheorem{ex}[theorem]{Example}

\usepackage{mathtools}

\begin{document}
\title[A nullstellensatz for pde]{A nullstellensatz for linear partial differential equations with polynomial coefficients}

\author{J. Cimpri\v c}

\begin{abstract}
In this paper an \textit{equation} means a homogeneous linear partial differential equation in $n$ unknown functions
of $m$ variables  which has real or complex polynomial coefficients. The \textit{solution set} consists of all $n$-tuples of real or complex analytic 
functions that satisfy the equation. For a given system of  equations we would like to characterize its \textit{Weyl closure},
i.e. the set of all equations that vanish on the solution set of the given system. It is well-known that in many special
cases the Weyl closure is equal to $B_m(\FF)N \cap A_m(\FF)^n$ where $\FF \in \{\RR,\CC\}$, the algebra $A_m(\FF)$ (respectively $B_m(\FF)$)
consists of all linear partial differential operators with coefficients in $\FF[x_1,\ldots,x_m]$ (respectively $\FF(x_1,\ldots,x_m)$)
and $N$ is the submodule of $A_m(\FF)^n$ generated by the given system. Our main result is that this formula holds in general.
In particular, we do not assume that the module $A_m(\FF)^n/N$ has finite rank which used to be a standard assumption.
Our approach works also for the real case which was not possible with previous methods. Moreover, our proof is constructive
as it depends only on the Riquier-Janet theory.
\end{abstract}

\thanks{Address: University of Ljubljana, Faculty of Mathematics and Physics, Department of Mathematics, Jadranska 21, SI-1000 Ljubljana, Slovenia}

\thanks{E-mail: cimpric@fmf.uni-lj.si}

\thanks{Supported by grant P1-0222 of the Slovenian Research Agency}


\keywords{algebraic theory of systems, symbolic methods for systems}

\date{\today}

\maketitle

\section{Introduction}

Let $\FF$ be either $\RR$ or $\CC$ and let $m$ and $n$ be integers. 
A homogeneous linear partial differential equation with polynomial coefficients 
in $n$ unknown functions $u^1,\ldots,u^n$ of $m$ variables $x_1,\ldots,x_m$ can be written as
$$p_1[u^1]+\ldots+p_n[u^n]=0$$
where linear partial differential operators $p_1,\ldots,p_n$ have polynomial coefficients;
in other words,  $p_1,\ldots,p_n$ belong to the Weyl algebra $A_m(\FF)$ which is generated by $x_1,\ldots,x_m$ and $\frac{\partial}{\partial x_1},\ldots,\frac{\partial}{\partial x_m}$.
Its solution at point $(x_1^0,\ldots,x_m^0) \in \FF^n$ is an $n$-tuple of convergent power series in $x_1-x_1^0,\ldots,x_m-x_m^0$ that satisfy the equation. 
The solution set consists of all solutions at all points of  $\FF^n$. 

The aim of this paper is to prove a nullstellensatz type result for such equations.
Consider a system of $k$ equations
\begin{align}
\notag p_{11}[u^1]+\ldots+p_{1n}[u^n] &= 0 \\
\label{peq} \vdots \\
\notag p_{k1}[u^1]+\ldots+p_{kn}[u^n] &= 0 
\end{align}
We would like to determine when another equation
\begin{equation}
\label{qeq}
q_1[u^1]+\ldots+q_n[u^n]=0
\end{equation}
vanishes on the solution set of \eqref{peq}. 
Our main result is that this happens if and only if there exists a nonzero polynomial $w \in \FF[x_1,\ldots,x_m]$
and a $k$-tuple of linear partial differential operators $(h_1,\ldots,h_k) \in A_m(\FF)^k$ such that the following matrix equation is true 
\begin{equation}
w \left[ \begin{array}{ccc}
q_1 & \ldots & q_n
\end{array} \right]=
\left[ \begin{array}{ccc}
h_1 & \ldots & h_k
\end{array} \right]
\left[ \begin{array}{ccc}
p_{11} & \ldots & p_{1n} \\
\vdots & & \vdots \\
p_{k1} & \ldots & p_{kn}
\end{array} \right]
\end{equation}

The set of all equations \eqref{qeq} that vanish on the solution set of the system \eqref{peq} is usually called
the \textit{Weyl closure} of \eqref{peq}. Let $N$ be the submodule of $A_m(\FF)^n$ that is generated by the
rows of the $p_{ij}$ matrix. Our result can be rephrased as follows: the Weyl closure of the system \eqref{peq}
is equal to $$\FF(x_1,\ldots,x_m)N \cap A_m(\FF)^n.$$

For constant coefficients our main result follows from \cite[Examples 1.13 and 1.13 (real), Assumption 2.55, Theorems 2.61 and 4.54]{oberst}.
Note that \cite{oberst} also covers other notions of solution which is further developed in \cite{shankar}.
For holonomic systems (with $\FF=\CC$) our main result follows from \cite[Proposition 2.1.9]{tsaiphd}.
This result uses global solutions instead of our local solutions. We will discuss it in subsection \ref{subsec52}.

The proof of our main result uses Riquier-Janet theory. Riquier bases are Weyl algebra analogues of Gr\" obner bases
while Janet's algorithm is an analogue of Buchberger's algorithm. Riquier existence theorems are
generalizations of the Cauchy-Kovalevskaya theorem. For a recent survey of this theory, see \cite[Chapter 4]{zerz}.

\section{Preliminaries}
\label{sec2}

Let $\FF$ be either $\RR$ or $\CC$. For every $m \in \NN$,\footnote{$\NN=\{0,1,2,\ldots\}$ is the set of natural numbers.} 
the \textit{Weyl algebra} $A_m(\FF)$ is the $\FF$-algebra with generators $x_1,\ldots,x_m,D_1,\ldots,D_m$ and relations
$x_i x_j=x_j x_i$, $D_i D_j = D_j D_i$ and $D_j x_i -x_i D_j=\varepsilon_{ij} \cdot 1$ for all $i,j=1,\ldots,m$,
where $\varepsilon_{ij}=1$ if $i=j$  and $\varepsilon_{ij}=0$ if $i \ne j$.
Clearly, $A_m(\FF)$ is a left module over $\FF[\x]:=\FF[x_1,\ldots,x_m]$.
We will also need its localization $B_m(\FF):=(\FF[\x]\setminus \{0\})^{-1} A_m(\FF)$
which is a left vector space over $\FF(\x):=\FF(x_1,\ldots,x_m)$.
It is well-known that $A_m(\FF)$ and $B_m(\FF)$ are Noetherian domains, 
see e.g. \cite[pp. 19--20]{mcr}, (which implies the Ore property by \cite[pp. 46--47]{mcr}).
For every $n \in \NN$, the left $A_m(\FF)$-module $A_m(\FF)^n$ and the left $B_m(\FF)$-module $B_m(\FF)^n$ are also Noetherian. 
For additional ring-theoretic information on $A_m(\FF)$ and $B_m(\FF)$ see \cite{stafford, quadrat}.

An element of $B_m(\FF)^n$ is a \textit{derivative} if it is
of the form $\delta_\alpha^i := D^{\alpha}\mathbf{e}_i$ where $\alpha=(\alpha_1,\ldots,\alpha_m) \in \NN^m$,
$D^\alpha:=D_1^{\alpha_1} \cdots D_m^{\alpha_m}$ and $\mathbf{e}_i=(\varepsilon_{i1},\ldots,\varepsilon_{in})$
is the $i$-th standard basis vector of $B_m(\FF)^n$.
The set of all derivatives will be denoted by $\Delta$.
Every element $\mathbf{p} \in B_m(\FF)^n$ can be converted into a \textit{standard form}, i.e. it can be expressed uniquely 
as a left $\FF(\x)$-linear combination of different derivatives.
We write $\cf(\mathbf{p})(\delta)$ for the coefficient of $\mathbf{p}$ at $\delta \in \Delta$, 
so $\mathbf{p}=\sum_{\delta \in \Delta} \cf(\mathbf{p})(\delta)\delta$.
The \textit{standard ranking} is a linear ordering $\prec$ of the set $\Delta$ which is defined by
$$\delta_\alpha^i \prec \delta_\beta^j \Leftrightarrow (\vert \alpha \vert,i,\alpha_1,\ldots,\alpha_{m-1}) \le_{\lex}
(\vert \beta \vert,j,\beta_1,\ldots,\beta_{m-1})$$
where  $\vert \alpha\vert:=\alpha_1+\ldots+\alpha_m$ and $\le_{\lex}$ is the usual lexicographic ordering.
It determines the notions of the leading coefficient $\hc \mathbf{p} $ and the highest derivative $\hd \mathbf{p}$
of an element $\mathbf{p} \in B_m(\FF)$. If $\hd \mathbf{p} =\delta_\alpha^i$ we define the \textit{degree} of $\mathbf{p}$ by $\deg \mathbf{p}:= \vert \alpha\vert$.
The standard ranking satisfies the following property (which defines a \textit{ranking}):
if $\delta_\alpha^i \prec \delta_\beta^j$ for some $\alpha,\beta \in \NN^m$ and $i,j=1,\ldots,n$,
then $\delta_{\alpha+\gamma}^i \prec \delta_{\beta+\gamma}^j$ for all $\gamma \in \NN^m$.
The standard ranking belongs to several interesting classes of rankings that appear in the literature
(positive rankings, orderly rankings, Riquier rankings); see \cite{rustphd}.
Similar remarks apply to elements of $A_m(\FF)^n$.

For a given point $\x^0=(x_1^0,\ldots,x_m^0) \in \FF^m$ we will write $\FF[[\x-\x^0]]$ for the set of all formal power series in 
$x_1-x_1^0,\ldots,x_m-x_m^0$. We say that a formal power series is \textit{convergent} if it has a nonzero convergence radius.
In this case it defines an analytic function on a ball around $\x^0$.
Every element $p \in B_m(\FF)$ which is defined (i.e. whose coefficients are defined) at $\x^0$ induces in a natural way 
a mapping $u \mapsto p[u]$ from $\FF[[\x-\x^0]]$ to itself which respects convergence.
Similarly, every element $\mathbf{p} =(p_1,\ldots,p_n) \in B_m(\FF)^n$ which is defined at $\x^0$
induces a mapping from $\FF[[\x-\x^0]]^n$ to $\FF[[\x-\x^0]]$ by
$\mathbf{u}=(u^1,\ldots,u^n) \mapsto \mathbf{p}[\mathbf{u}]:=p_1[u^1]+\ldots+p_n[u^n]$.

For every finite subset $\{\mathbf{p}_1,\ldots,\mathbf{p}_k\}$ of $B_m(\FF)^n$, we have a \textit{system} 
\begin{equation}
\label{sys}
\mathbf{p}_1[\mathbf{u}]=\ldots=\mathbf{p}_k[\mathbf{u}]=0
\end{equation}
of partial differential equations corresponding to it. 
We say that an element $\mathbf{u} \in \FF[[\x-\x^0]]^n$
is a \textit{formal solution} of system \eqref{sys} at point $\x^0 \in \FF^m$
if all $\mathbf{p}_1,\ldots,\mathbf{p}_k$ are defined at $\x^0$ and
$\mathbf{u}$ satisfies $\eqref{sys}$ in $\FF[[\x-\x^0]]$.
If a formal solution at $\x^0$ is convergent, then the corresponding analytic function 
solves the system on a ball around $\x^0$.
If two finite subsets of $B_m(\FF)^n$ generate the same submodule of $B_m(\FF)^n$ then the
corresponding systems are \textit{equivalent}, i.e. they have the same formal and the same analytic solutions
at every point $\x^0$ from some open dense subset of $\FF^m$.

We will now summarize the Riquier-Janet theory. Let $N$ be a submodule of $B_m(\FF)^n$ and let $\mathcal{N}$ be a finite
generating set of $N$. A procedure called the \textit{Janet's algorithm}\footnotemark
\footnotetext{The original reference is \cite{janet}. 
A recent monography is \cite[Section 2.1]{robertz}. 
We use the terminology from \cite[chapter 5]{rustphd}.}
transforms $\mathcal{N}$ into a better finite generating set
$\mathcal{M}$ that we call a \textit{Riquier basis}. The idea is to transform each element
$a \delta+L \in \mathcal{N}$ (where $a \in \FF(\x)$, $\delta \in \Delta$ and $\hd L \prec \delta$)
into a substitution rule $\delta \mapsto -a^{-1}L$ that is used to reduce other elements of $\mathcal{N}$. We must also ensure that
by differentiating the substitution rules for $\delta_\alpha^i$ and $\delta_\beta^i$ (when they exist) we get only one substitution rule for $\delta_{\alpha+\beta}^i$.
By definition, all elements of $\mathcal{M}$ are monic.
The system  corresponding to $\mathcal{M}$ is equivalent to the system corresponding to $\mathcal{N}$ but it is much easier to solve.

The procedure to formally solve the system corresponding to $\mathcal{M}$ is given by the Formal Riquier Existence Theorem. 
The idea is to split the set $\Delta$ into two parts, the set of \textit{principal derivatives} $\prin \mathcal{M}$ which is defined by 
$$\prin \mathcal{M} := \{\delta \in \Delta \mid \delta=D^\alpha \hd \mathbf{f} \text{ for some } \alpha \in \NN^m \text{ and some } \mathbf{f} \in \mathcal{M}\}$$
and the set  of \textit{parametric derivatives} $\ppar \mathcal{M}:=\Delta \setminus \prin \mathcal{M}$. Pick a point $\x^0$ in which all
elements of $\mathcal{M}$ are defined. For each parametric derivative, we can specify an initial condition in $\x^0$. We then use the equations from $\mathcal{M}$
to (uniquely) compute the values of principal derivatives at $\x^0$ and thus obtain a formal solution of the system corresponding to $\mathcal{M}$.
If the set $\ppar \mathcal{M}$ is empty, then the system corresponding to $\mathcal{M}$ has only the trivial solution.
We refer the reader to \cite[Theorem 2]{rustart} or to \cite{rustphd} for the details,  including the details about Riquier bases.

Finally, the Analytic Riquier Existence Theorem states that the formal solution of the system defined by $\mathcal{M}$
is convergent if all initial determinations are convergent. Recall that for each $i=1,\ldots,n$ the \textit{initial determination} of $u^i$
is the formal power series with support $\{\alpha \in \NN^m \mid \delta_\alpha^i \in \ppar \mathcal{M}\}$ and with coefficients determined by the
initial conditions. We refer the reader to \cite[Chapter VIII]{ritt} for the proof.  The original reference is \cite{riquier}.
We do not use the full generality of this result since we only work with linear partial differential equations.
Reference \cite{rustphd} claims a generalization of the original result from Riquier to orderly rankings but
this has been disputed in \cite{lemaire}. This is not a problem for us because  the standard ranking is a Riquier ranking.

\section{A technical result}

The aim of this section is to prove the following technical result. For every integer $s$ we 
write $I_s=\{\alpha \in \NN^m \mid \vert \alpha \vert \le s\}$ and $\Delta_s=\{\delta_\alpha^i \in \Delta \mid \alpha \in I_s,i=1,\ldots,n\}$.

\begin{prop}
\label{prop1}
Let $\mathcal{M}$ be a Riquier basis in $B_m(\FF)^n$. 
Let $s_0$ be the maximum of degrees of all elements from $\mathcal{M}$. (Recall that degrees are defined with respect to the standard ranking.) 
We claim that for every integer $s \ge s_0$, every point $\x^0 \in \FF^m$ in which all elements of $\mathcal{M}$ are defined
(note that all $D^\beta \mathbf{p}$ are defined in every point in which $\mathbf{p}$ is defined)
and every $c \in \FF^{\Delta_s}$ the following are equivalent.
\begin{enumerate}
\item There exists a convergent $\mathbf{u} \in \FF[[\x-\x^0]]^n$ such that 
\begin{enumerate}
\item[(a)]  $\mathbf{p}[\mathbf{u}]=0$ for every $\mathbf{p} \in \mathcal{M}$ and 
\item[(b)] $\delta[\mathbf{u}](\x^0)=c(\delta)$ for every $\delta \in \Delta_s$.
\end{enumerate}
\item For every $\mathbf{p} \in \mathcal{M}$ and every $\beta \in I_{s-\deg \mathbf{p}}$, we have that
$$\sum_{\delta \in \Delta_s} \cf(D^\beta \mathbf{p})(\delta)\big|_{\x^0} c(\delta)=0.$$
\end{enumerate}
\end{prop}

\begin{proof}
To prove that (1) implies (2) we multiply (a) with $D^{\beta}$, convert into standard form, insert $\x^0$ and finally apply (b).
Suppose now that (2) is true. If $\ppar \mathcal{M}$ is empty, then $\mathcal{M}$ must contain elements with highest derivatives $\delta_0^i=\mathbf{e}_i$ for all $i$.
Then assumption (2) implies that $c(\delta)=0$ for every $\delta \in \Delta_s$. Now the trivial solution satisfies (1). 
If $\ppar \mathcal{M}$ is nonempty, we can proceed as in the Formal Riquier Existence Theorem. 
We compute the formal solution $\mathbf{u}=(u^1,\ldots,u^n)$ of the system defined by $\mathcal{M}$ that satisfies the following initial conditions
$$\delta[\mathbf{u}](\x^0):=\left\{ \begin{array}{cc} 
c(\delta) & \text{ if } \delta \in \ppar\mathcal{M} \cap \Delta_s \\ 
0 & \text{ if } \delta \in \ppar\mathcal{M} \setminus \Delta_s 
\end{array} \right.$$
By construction, $\mathbf{u}$ satisfies (a). Let us show now that $\mathbf{u}$ is analytic.
For each $i=1,\ldots,n$, the initial determination of $u^i$, i.e. the formal power series
$$\sum_{\alpha \in \NN^m \atop \delta_\alpha^i \in \ppar\mathcal{N}} \frac{D^\alpha u^i(\x^0)}{\alpha!}(\x-\x^0)^\alpha=
\sum_{\alpha \in \NN^m \atop \delta_\alpha^i \in \ppar\mathcal{N}}\frac{\delta_\alpha^i[\mathbf{u}](\x^0)}{\alpha!}(\x-\x^0)^\alpha
=\!\!\sum_{\alpha \in \NN^m \atop \delta_\alpha^i \in \ppar\mathcal{N} \cap \Delta_s}\!\!\frac{c(\delta_\alpha^i)}{\alpha!}(\x-\x^0)^\alpha$$
is a polynomial. By the Analytic Riquier Existence Theorem\footnotemark\footnotetext{See the last paragraph of Section \ref{sec2}.}
it follows that the formal power series for $\mathbf{u}$ is convergent. 
It remains to show that $\mathbf{u}$ satisfies (b). By construction, we already know that
\begin{equation}
\label{prineq}
\delta[\mathbf{u}](\x^0)=c(\delta)
\end{equation}
holds for every $\delta \in \ppar\mathcal{M} \cap \Delta_s$. We claim that \eqref{prineq} also holds for every $\delta \in \prin \mathcal{M} \cap \Delta_s$.
We will prove this claim by induction.
Pick any $\delta_\alpha^i \in \prin \mathcal{M} \cap \Delta_s$ and assume that \eqref{prineq} holds for all $\delta \prec \delta_\alpha^i$.
By the definition of $\prin \mathcal{M}$ there exists $\mathbf{p} \in \mathcal{M}$ and $\beta \in \NN^m$ such that $\delta_\alpha^i =D^\beta \hd \mathbf{p}$.
 Now assumption (2) implies that 
$$\sum_{\delta \prec \delta_\alpha^i} \cf(D^\beta \mathbf{p})(\delta)\big|_{\x^0} c(\delta)+\cf(D^\beta \mathbf{p})(\delta_\alpha^i)\big|_{\x^0}c(\delta_\alpha^i)=0.$$
On the other hand, by multiplying the equation $\mathbf{p}[\mathbf{u}]=0$ with $D^\beta$, converting into the standard form and inserting $\x^0$ we obtain that
$$\sum_{\delta \prec \delta_\alpha^i} \cf(D^\beta \mathbf{p})(\delta)\big|_{\x^0} \delta[\mathbf{u}](\x^0)+
\cf(D^\beta \mathbf{p})(\delta_\alpha^i)\big|_{\x^0}\delta_\alpha^i[\mathbf{u}](\x^0)=0.$$
Now, the induction hypothesis implies that
$$\cf(D^\beta \mathbf{p})(\delta_\alpha^i)\big|_{\x^0}c(\delta_\alpha^i)=
\cf(D^\beta \mathbf{p})(\delta_\alpha^i)\big|_{\x^0}\delta_\alpha^i[\mathbf{u}](\x^0)$$
The fact that all elements of $\mathcal{M}$ are monic implies that $\cf(D^\beta \mathbf{p})(\delta_\alpha^i)\big|_{\x^0}=1$, so
$$c(\delta_\alpha^i)=\delta_\alpha^i[\mathbf{u}](\x^0)$$
which completes our induction and proves the claim.
\end{proof}

\section{Proof of the main result}

We will prove a slight generalization of the promised result. Namely, that for every nonempty open set $U \subseteq \FF^m$
we can restrict our solution set from a subset of $\bigcup_{\x^0 \in \FF^m} \FF[[\x-\x^0]]$ to a subset of $\bigcup_{\x^0 \in U} \FF[[\x-\x^0]]$.
 We will use several times that a nonzero polynomial from $\FF[\x]$ cannot
vanish on a nonempty open subset of $\FF^m$. It follows that the zero set of a nonzero polynomial has the property that 
its relative complement in any nonempty open subset of $\FF^m$ is dense in that subset.

We will need the following auxiliary observation:

\begin{lem}
\label{lem1}
Pick $t \in \NN$ and let $\langle\cdot,\cdot\rangle$ be the standard inner product on $\FF^t$.
We claim that for every $\mathbf{g}_1,\ldots,\mathbf{g}_k,\mathbf{f} \in \FF(\x)^t$ 
the following are equivalent:
\begin{enumerate}
\item There exists a nonempty open subset $W \subseteq \FF^m$ on which all $\mathbf{g}_1,\ldots,\mathbf{g}_k,\mathbf{f}$ are defined
such that for every $\x^0 \in W$ and for every $\mathbf{c} \in \FF^t$ which satisfy
$\langle \mathbf{g}_1(\x^0),\mathbf{c} \rangle=\ldots=\langle \mathbf{g}_k(\x^0),\mathbf{c} \rangle=0$ we have that 
$\langle \mathbf{f}(\x^0),\mathbf{c} \rangle=0$. 
\item $\mathbf{f} \in \FF(\x) \mathbf{g}_1+\ldots+\FF(\x) \mathbf{g}_k$.
\end{enumerate}
\end{lem}

\begin{proof}
If (2) is true then $\mathbf{f}=\sum_{j=1}^k h_j \mathbf{g}_j$ for some $h_j \in \FF(\x)$.
Let $p$ be the product of denominators of all $h_j$ and of all components of  $\mathbf{f}$ and of all components of all $\mathbf{g}_j$.
The set $W:=\{\x^0 \in \FF^m \mid p(\x^0) \ne 0\}$ is an open subset of $\FF^m$ on which $\mathbf{f}$ and all $\mathbf{g}_j$  are defined.
Pick any $\x^0 \in W$ and any $\mathbf{c} \in \FF^t$ such that 
$\langle \mathbf{g}_j(\x^0),\mathbf{c} \rangle=0$ for all $j=1,\ldots,k$ and note that
$\langle \mathbf{f}(\x^0),\mathbf{c} \rangle=\sum_{j=1}^k h_j(\x^0) \langle \mathbf{g}(\x^0), \mathbf{c} \rangle=0$.
So, (1) is true.

Suppose now that (1) is true. Let $G$ be the matrix with rows $\mathbf{g}_1,\ldots,\mathbf{g}_k$ and let 
$\mathbf{v} \in \FF(\x)^t$ be a column vector such that $G\,\mathbf{v}=0$. We claim that $\mathbf{f}\,\mathbf{v}=0$.
We may assume that $\mathbf{v} \in \FF[\x]^t$. Pick any $\x^0 \in W$,
write $\mathbf{c}=\overline{\mathbf{v}(\x^0)}^T$ and note that 
$\langle \mathbf{g}_j(\x^0),\mathbf{c} \rangle=0$ for all $j=1,\ldots,k$.
By (1), it follows that $\mathbf{f}(\x^0) \mathbf{v}(\x^0)=\langle \mathbf{f}(\x^0),\mathbf{c} \rangle=0$.
We proved that $\mathbf{f}\,\mathbf{v}$ vanishes on $W$. As $\mathbf{f}$ is defined on $W$, it follows that
the numerator of $\mathbf{f}\,\mathbf{v}$ vanishes on $W$. Thus $\mathbf{f}\,\mathbf{v}=0$ in $\FF(\x)$.
Now we use a standard linear algebra trick. We define a $\FF(\x)$-linear function 
$$\phi \colon \FF(\x)^k \to \FF(\x), \quad \phi(\mathbf{u})=\left\{\begin{array}{cc}
\mathbf{f}\,\mathbf{v} & \text{ if } \mathbf{u}=G\,\mathbf{v} \text{ for some } \mathbf{v} \in \FF[\x]^t\\
0  & \text{ if } \mathbf{u} \ne G\,\mathbf{v} \text{ for every } \mathbf{v} \in \FF[\x]^t
\end{array} \right.$$
Since $G\,\mathbf{v}=0$ implies $\mathbf{f}\,\mathbf{v}=0$, $\phi$ is well-defined.
By construction, we have that $\phi(G\,\mathbf{v})=\mathbf{f}\,\mathbf{v}$ for every $\mathbf{v} \in \FF(\x)^t$.
It follows that $\mathbf{f}=\sum_{j=1}^k\phi(\mathbf{e}_j) \mathbf{g}_j$ 
where $\mathbf{e}_j$ is the $j$-th standard basis vector of $\FF(\x)^k$. So, (2) is true.
\end{proof}

We are now ready for the proof of our main result.

\begin{thm}
\label{mainthm}
Let $U$ be a nonempty open subset of $\FF^m$.
For every submodule $N$ of $A_m(\FF)^n$ and every element $\mathbf{q} \in A_m(\FF)^n$, the following are equivalent:
\begin{enumerate}
\item[(1)] Every convergent $\mathbf{u} \in \bigcup_{\x^0 \in \FF^m} \FF[[\x-\x^0]]^n$ which solves $\mathbf{p}[\mathbf{u}]=0$ for all $\mathbf{p} \in N$,
also solves $\mathbf{q}[\mathbf{u}]=0$.
\item[(1')] Every convergent $\mathbf{u} \in \bigcup_{\x^0 \in U} \FF[[\x-\x^0]]^n$ which solves $\mathbf{p}[\mathbf{u}]=0$ for all $\mathbf{p} \in N$,
also solves $\mathbf{q}[\mathbf{u}]=0$.
\item[(2)] There exists a nonzero $w \in \fx$ such that $w \, \mathbf{q} \in N$.
\end{enumerate}
\end{thm}

\begin{proof}
Clearly, (1) implies (1'). 

To show that (2) implies (1), one must show that $(w \mathbf{q})[\mathbf{u}]=0$ implies $\mathbf{q}[\mathbf{u}]=0$.
This follows from continuity of analytic functions (and their derivatives) and the fact that the complement of the zero set of $w$ is dense
in any ball around $\x^0$.

To show that (1') implies (2), note first that $\FF(\x) N$ is a submodule  of $B_m(\FF)^n$. 
Pick a Riquier basis $\mathbf{p}_1,\ldots,\mathbf{p}_k$ of $\FF(\x)N$ and write
$$s=\max\{\deg \mathbf{q},\deg \mathbf{p}_1,\ldots,\deg \mathbf{p}_k\}, \quad t= \card \Delta_s.$$
The standard ranking identifies $\Delta_s$ with $\{1,\ldots,t\}$, $\FF^{\Delta_s}$ with $\FF^t$ and $\FF(\x)^{\Delta_s}$ with $\FF(\x)^t$.
Let $\cf_s \colon B_m(\FF)^n \to \FF(\x)^{\Delta_s}$ be the compositum of $\cf \colon B_m(\FF) \to \FF(\x)^{\Delta}$ 
with the restriction map $\FF(\x)^{\Delta} \to \FF(\x)^{\Delta_s}$.

We claim that elements $\mathbf{f}:=\cf_s(\mathbf{q})$ and $\mathbf{g}_{j,\beta}:=\cf_s(D^\beta \mathbf{p}_j)$ (for $j=1,\ldots,k$ and $\beta \in I_{s-\deg \mathbf{p}_j}$)
satisfy part (1) of Lemma \ref{lem1}. The set $W$ of all $\x^0 \in U$ in which $\mathbf{f}$ and all $\mathbf{g}_{j,\beta}$ are defined is clearly nonempty and open.
Pick any $\x^0 \in W$ and any $\mathbf{c}=(\overline{c(\delta)})_{\delta \in \Delta_s} \in \FF^{\Delta_s}$ 
such that $\langle \mathbf{g}_{j,\beta}(\x^0),\mathbf{c} \rangle=0$ for all $j$ and $\beta$.
Note that part (2) of Proposition \ref{prop1} is satisfied since 
$\sum_{\delta \in \Delta_s} \cf(D^\beta \mathbf{p}_j)(\delta)\big|_{\x^0} c(\delta)=
\sum_{\delta \in \Delta_s} \mathbf{g}_{j,\beta}(\delta)\big|_{\x^0} \overline{\mathbf{c}(\delta)}=
\langle \mathbf{g}_{j,\beta}\big|_{\x^0}, \mathbf{c} \rangle =0$
for every $\mathbf{p}_j$ and every $\beta \in I_{s-\deg \mathbf{p}_j}$.
By part (1) of Proposition \ref{prop1}, there exists a convergent $\mathbf{u} \in \FF[[\x-\x^0]]^n$ such that
$\mathbf{p}_j[\mathbf{u}]=0$ for every $j=1,\ldots,k$ and $\delta[\mathbf{u}](\x^0)=c(\delta)$ for every $\delta \in \Delta_s$.
It follows that $\mathbf{p}[\mathbf{u}]=0$ for every $\mathbf{p} \in N$. (This requires a continuity argument as above, 
as $\mathbf{p} \in \sum_{j=1}^k B_m(\FF) \mathbf{p}_j$ implies only that $(z \mathbf{p})[\mathbf{u}]=0$ for some nonzero $z \in \fx$.)
By assumption (1') it follows that $\mathbf{q}[\mathbf{u}]=0$. If we insert $\x^0$ and use that
$\delta[\mathbf{u}](\x^0)=c(\delta)$, we get that $\langle \mathbf{f}(\x^0),\mathbf{c} \rangle=0$. 
This proves the claim. Now, Lemma \ref{lem1} implies that 
$$\mathbf{f} \in \sum_{j=1}^k \sum_{\beta \in I_{s-\deg \mathbf{p}_j}} \FF(\x) \mathbf{g}_{j,\beta}.$$
Since $\sum_{\delta \in \Delta_s} \mathbf{f}(\delta)\delta=\mathbf{q}$ and $\sum_{\delta \in \Delta_s} \mathbf{g_{j,\beta}}(\delta)\delta=\mathbf{D^\beta \mathbf{p}}_j$, we obtain
$$\mathbf{q} \in \sum_{j=1}^k \sum_{\beta \in I_{s-\deg \mathbf{p}_j}} \FF(\x) D^\beta \mathbf{p}_j \subset \sum_{j=1}^k B_m(\FF)\mathbf{p}_j=\FF(\x)\cdot N $$
which implies (2).
\end{proof}

\section{Comments and examples}

\subsection{Simplifications in the $m=n=1$ case}

If $m=1$ then $B_m(\FF)$ is a principal left ideal domain by \cite[Theorem 1.5.9 (ii)]{mcr}.
If $n=1$ then every submodule of $B_m(\FF)^n$ is a left ideal of $B_m(\FF)$.
Therefore, if $m=n=1$ then every submodule of $B_m(\FF)^n$ is principal.
Let $I$ be a left ideal of $B_1(\FF)$ and let $p=\sum_{i=0}^{s_0} p_i(x)D^i$, where $p_{s_0} =1$, be its principal generator.
The set $\mathcal{I}=\{p\}$ is then a Riquier basis of $I$.
We have that $\Delta=\{D^n, n \in \NN\}$ and its standard ranking comes from the usual ordering of $\NN$.
We can decompose $\Delta$ into $\ppar \mathcal{I}=\{D^n \mid n=0,\ldots,s_0-1\}$ and $\prin \mathcal{I}=\{D^n \mid n \ge s_0\}$.
Pick a point $\x^0$ in which all coefficients of $p$ are defined.
The Analytic Riquier Existence Theorem reduces to the well-known fact that the initial value problem 
$\sum_{i=0}^{s_0} p_i(x)u^{(i)}(x)=0$, $(u(x^0),u'(x^0),\ldots, u^{(s_0-1)}(x^0))=\mathbf{c}$ 
has a unique convergent power series solution for each $\mathbf{c} \in \FF^{s_0}$.
Apart from these simplifications the length of the proof of Theorem \ref{mainthm} in the $m=n=1$ case remains the same as in the general case.

\subsection{Nonsingular points}\label{subsec52}

We define a singular point of system \eqref{peq} as a point (in $\FF^m$) that belongs to its singular locus, see \cite[Definition 2.1.3]{tsaiphd}.
If $m=n=1$ this coincides with the usual definition. 
System \eqref{peq} has a nonsingular point if and only if the module $A_m(\FF)^n/N$, where submodule $N$ is generated by the rows of the $[p_{ij}]$ matrix,
has finite rank, see \cite[Lemma 2.1.5]{tsaiphd}. 
Note that the set of all nonsingular points is open in $\FF^n$.
Proposition \ref{regpoint} strengthens Theorem \ref{mainthm} in a special case.

\begin{prop}
\label{regpoint}
Let $N$ be as above. Suppose that $\FF=\CC$ and system \eqref{peq} has a nonsingular point $\x^0$. 
Let $U$ be a nonempty simply connected open subset of the set of all nonsingular points.
Then the following are equivalent for every $\mathbf{q} \in A_m(\FF)^n$:
\begin{enumerate}
\item[(1'')] Every convergent $\mathbf{u} \in \FF[[\x-\x^0]]^n$ which solves $\mathbf{p}[\mathbf{u}]=0$ for all $\mathbf{p} \in N$,
also solves $\mathbf{q}[\mathbf{u}]=0$.
\item[(2)] There exists a nonzero $w \in \fx$ such that $w \, \mathbf{q} \in N$.
\item[(3)] Every $n$-tuple $\mathbf{u}$ of analytic functions on $U$ which solves $\mathbf{p}[\mathbf{u}]=0$ for all $\mathbf{p} \in N$,
also solves $\mathbf{q}[\mathbf{u}]=0$.
\end{enumerate}
\end{prop}

\begin{proof}
Pick an open ball $B$ around $\x^0$ in the set of all nonsingular points. By the Cauchy-Kovalevskaya-Kashiwara theorem\footnotemark
\footnotetext{This version is from \cite[Theorem 2.1.8]{tsaiphd} or \cite[Section 4]{oaku}.
The original reference is Kashiwara's master's thesis  \cite[Theorem 2.3.1]{kashiwara}.},
the dimension of the space of all analytic solutions on $B$ is finite and equal to the rank of $A_m(\FF)^n/N$.  It follows that every 
convergent power series solution at $\x^0$ comes from some analytic solution on $B$.
Therefore, the equivalence of (2) and (1'') follows from the equivalence of (2) and (3).
The equivalence of (2) and (3) is a reformulation of \cite[Proposition 2.1.9]{tsaiphd} (which is also a corollary of the Cauchy-Kovalevskaya-Kashiwara theorem).
\end{proof}

Proposition \ref{regpoint} also holds for some singular points $\x^0$ and some open $U$ that are not simply connected (see Example \ref{singex})
but not for all of them (see Example \ref{singex2}).

\begin{ex}
\label{singex}
Take $\FF=\CC$, $U=\FF \setminus \{0\}$, $x^0=0$ and $p=x^2 D^2-2 x D+2$. Clearly $x^0$ is a singular point of $p$ and $U$ is not simply connected.
We claim that (1''), (2) and (3) are equivalent for every $q \in A_1(\FF)$. Suppose that $q \in A_1(\FF)$ satisfies either (1'') or (3).
Every convergent power series solution at $x^0$ and every analytic solution on $U$
of $p[u]=0$ are of the form $u=c_1 x+c_2 x^2$. Therefore, $q[x]=q[x^2]=0$. 
It follows that $q$ also satisfies (1') of Theorem \ref{mainthm} and so (2) is true.
The converse is clear.
\end{ex}

\begin{ex}
\label{singex2}
Take $U=\FF \setminus \{0\}$, $x^0=0$ and $p=x^2 D^2-x D+\frac{3}{4}$ then a general solution of $p[u]=0$ is $u=c_1 \sqrt{x}+c_2 x \sqrt{x}$. 
Therefore, $p[u]=0$ has no convergent power series solution at $x^0$ and no analytic solution on $U$ which
implies that (1'') and (3) are trivially true for all $q$. On the other hand, (2) is false for some $q$.
\end{ex}

\subsection{Generic solution}

Let $N$ be submodule of $A_m(\FF)^n$ generated by the rows of the $[p_{ij}]$ matrix of system \eqref{peq} and let $M=A_m(\FF)^n/N$. 
Let $\pi \colon A_m(\FF)^n \to M$ be the canonical projection and let $y_i=\pi(\mathbf{e}_i), i=1,\ldots, n$ be the projections of the standard basis of $A_m(\FF)^n$.
We will call $(y_1,\ldots,y_n)$ the \textit{generic solution} of system \eqref{peq}, see \cite[Definition 3.5.1 and Example 3.5.2]{quadrat2}.
To show that the generic solution is indeed a solution, note that by the definition of $N$, $\sum_{j=1}^n p_{ij} \mathbf{e}_j \in N$ for every $i=1,\ldots,k$. 
It follows that $\sum_{j=1}^n p_{ij} y_j =\sum_{j=1}^n p_{ij} \pi(\mathbf{e}_i) =\pi(\sum_{j=1}^n p_{ij} \mathbf{e}_j) =0$ for every $i=1,\ldots,k$ as desired.

All solutions of system \eqref{peq} can be obtained by specializing the generic solution, see \cite[Theorem 1.1.1]{quadrat2}.
Let us explain the details. For every $\mathbf{x}^0 \in \FF^n$ write $\mathcal{F}_{\mathbf{x}^0}$ for the abelian group of all 
convergent power series in $\FF[[\x-\x^0]]$. Note that $\mathcal{F}_{\x^0}$ has the structure of a left $A_m(\FF)$-module
in the obvious way. Let $\hmf(M,\mathcal{F}_{\x^0})$ be the set of all $A_m(\FF)$-module homomorphisms from $M$ to $\mathcal{F}_{\x^0}$.
For every $\varphi \in \hmf(M,\mathcal{F}_{\x^0})$, $(\varphi(y_1),\ldots,\varphi(y_n))$ is a solution of system \eqref{peq} at point $\x^0$
and every solution can be obtained this way.

We will now rephrase Theorem \ref{mainthm} in this new terminology. Note that every element $m \in M$ is of the form $m=\pi(\mathbf{q})=q_1 y_1+\ldots+q_n y_n$
for some $\mathbf{q}=(q_1,\ldots,q_n) \in A_m(\FF)^n$ where $(y_1,\ldots,y_n)$ is the generic solution.

\begin{cor}
Let $U$ be a nonempty open subset of $\FF^m$ and let $M$ be as above. For every element $m \in M$, the following are equivalent:
\begin{enumerate}
\item[(1)] For every $\x^0 \in \FF^m$ and every $\varphi \in \hmf(M,\mathcal{F}_{\x^0})$ we have that $\varphi(m)=0$.
\item[(1')] For every $\x^0 \in U$ and every $\varphi \in \hmf(M,\mathcal{F}_{\x^0})$ we have that $\varphi(m)=0$.
\item[(2)] There exists a nonzero $w \in \fx$ such that $w \, m=0$.
\end{enumerate}
\end{cor}

Proposition \ref{regpoint} can be rephrased similarly.

\subsection{Rapidly decreasing solutions}

Recall that a function is rapidly decreasing if it belongs to 
$\mathcal{S}:= \{f \in \mathcal{C}^{(\infty)}(\RR^m) \mid \sup_{x \in \RR^m} | x^\alpha D^\beta f(x) | < \infty$ for all $\alpha,\beta \in \NN^m\}$.
We define the $\mathcal{S}$-closure of system \eqref{peq} as the set of all equations \eqref{qeq} that vanish on every rapidly decreasing solution of system \eqref{peq}.
The $\mathcal{S}$-closure behaves very differently from the Weyl closure as the following example shows:

\begin{ex}
\label{sexample}
Let $q=D+x \in A_1(\RR)$ and $p=q^\ast q=(-D+x)(D+x)=-D^2+x^2-1$. (Recall that the standard involution on $A_m(\FF)$ is defined by $D_i^\ast=-D_i,x_j^\ast=x_j$ for every $i,j=1,\ldots,m$
and $\alpha^\ast=\bar{\alpha}$ for every $\alpha \in \FF$.) We claim that $q$ belongs to the $\mathcal{S}$-closure of $p$ but it does not belong to the Weyl closure of $p$.
The general solution of $q[u]=0$ is $u=c \, e^{-x^2/2}$ which is rapidly decreasing and the general solution of $p[u]=0$ is $u=c_1 e^{-x^2/2}+c_2 v$ where $v=e^{-x^2/2} \int e^{x^2} \, dx$ is not rapidly decreasing.
It follows that $q$ belongs to the $\mathcal{S}$-closure  of $p$. Since $q[v]=e^{x^2/2} \ne 0$, we have that $q$ does not belong to the Weyl closure of $p$.
\end{ex}

An important advantage of $\mathcal{S}$ is that it has an inner product and that $q^\ast$ is the adjoint of $q$ with respect to this inner product.
A disadvantage is that the $\mathcal{S}$-closure is often equal to $A_m(\FF)^n$ because often there is no rapidly decreasing solution.
Let $N'$ be the $\mathcal{S}$-closure of a submodule $N$ of $A_m(\FF)^n$. By using the inner product one can show that $N'$ is a real submodule
of $A_m(\FF)^n$ in the sense that if 
$$\sum_i \mathbf{p}_i^\ast \mathbf{p}_i= \sum_j (\mathbf{h}_j^\ast \mathbf{q}_j+\mathbf{q}_j^\ast \mathbf{h}_j) \quad ( \text{in }  A_m(\FF)^{n \times n})$$
for some $\mathbf{p}_i, \mathbf{h}_j \in A_m(\FF)^n$ and $\mathbf{q}_j \in N'$ then $\mathbf{p}_i \in N'$ for all $i$.
From the perspective of noncommutative real algebraic geometry (see \cite[Example 1.3 and Theorem 1.6]{c1} and \cite[Theorem 2]{c2})
it would be interesting to know when $N'$ is the the smallest real submodule of $A_m(\FF)^n$ which contains $N$ (i.e. when $N'$ is the real radical of $N$).

\end{document}